\documentclass[sn-mathphys,Numbered]{sn-jnl}
 
\usepackage{graphicx}%
\usepackage{multirow}%
\usepackage{amsmath,amssymb,amsfonts}%
\usepackage{amsthm}%
\usepackage{mathrsfs}%
\usepackage[title]{appendix}%
\usepackage{xcolor}%
\usepackage{textcomp}%
\usepackage{manyfoot}%
\usepackage{booktabs}%
\usepackage{algorithm}%
\usepackage{algorithmicx}%
\usepackage{algpseudocode}%
\usepackage{listings}%

\theoremstyle{thmstyleone}%
\theoremstyle{thmstyletwo}%

\theoremstyle{thmstylethree}%

\usepackage{lineno}
 
\theoremstyle{definition}
\newtheorem{thm}{Theorem}[section]
\newtheorem{dfn}[thm]{Definition}
\newtheorem{pro}[thm]{Problem}

\newtheorem{lem}[thm]{Lemma}
\newtheorem{exa}[thm]{Example}

\newcommand{\R}{\mathbb R}

\newcommand{\al}{\alpha}
\newcommand{\be}{\beta}
\newcommand{\ga}{\gamma}

\newcommand{\de}{\delta}
\newcommand{\De}{\Delta}
\newcommand{\la}{\lambda}

\newcommand{\si}{\sigma}
\newcommand{\ti}{\tilde}
\newcommand{\ep}{\varepsilon}

\newcommand{\SO}{\mathrm{SO}}
\newcommand{\Or}{\mathrm{O}}
\newcommand{\vol}{\mathrm{vol}} 
\newcommand{\sign}{\mathrm{sign}}

\newcommand{\bpd}[2]{\dfrac{\partial #1}{\partial #2}}

\begin{document}

\title[The strength of a geometric simplex]{The strength of a geometric simplex}


\author[1]{\fnm{Olga} \sur{Anosova}}\email{oanosova@liv.ac.uk}

\author*[1]{\fnm{Vitaliy} \sur{Kurlin}}\email{vkurlin@liv.ac.uk}
\equalcont{}


\affil*[1]{\orgdiv{School of Computer Science and Informatics}, \orgname{University of Liverpool}, \orgaddress{\street{Ashton street}, \city{Liverpool}, \postcode{L69 3BX}, \country{UK}}}




\abstract{
The basic input for many real objects is a finite cloud of unordered points.
The strongest equivalence between objects in practice is rigid motion in a Euclidean space.
A recent polynomial-time classification of point clouds required a Lipschitz continuous function that vanishes on degenerate simplices, while the usual volume is not Lipschitz.
We define the strength of any geometric simplex and prove its continuity under perturbations with explicit bounds for Lipschitz constants.}




\maketitle

\section{The importance of Lipschitz continuity for distinguishing mirror images under noise}
\label{sec:intro}

Many applications deal with point configurations or clouds of points obtained as edge pixels or feature points of objects across all scales from galaxies to molecules.
\medskip

Positions in a Euclidean space, such as atomic centers, are always uncertain due to measurement noise or thermal vibrations, see Feynman's lecture "Atoms in motion" \cite[chapter~1]{feynman2011lectures}. 
Molecular dynamics simulates trajectories of atomic clouds, which evolve in time. 
Machine learning tries to predict molecular properties that depend on atomic geometry.
These predicted outputs are expected to be independent of coordinate representations and remain stable under small perturbations of atomic positions.
\medskip

Though real objects are often symmetric, their noisy representations deviate from ideal symmetry.
For example, a narrow triangle can degenerate to a straight line and evolve to a mirror image of the original triangle of opposite (sign of) orientation.
\medskip


This discontinuity challenge of traditional representations for finite and periodic point sets was formalized and extended to the geo-mapping problem \cite[Problem~1.4.5]{anosova2025geometric}, which aims to continuously parameterize spaces of real data objects under practical equivalences by complete invariants similar to geographic coordinates on Earth.
\medskip

The space of triangles (clouds of 3 unordered points) under \emph{isometry} (any distance-preserving transformation) can be parameterized by inter-point distances $a,b,c$ as 
$$\{(a,b,c)\in\R^3 \mid 0<a\leq b\leq c\leq a+b\}\subset\R^3.$$
The emerging area of Geometric Data Science \cite{anosova2025geometric} aims to develop such continuous parameterizations for all spaces of real data under isometry and other equivalences.
\medskip

A crucial step towards such parameterizations is to guarantee their continuous change under noise.
The classical $\ep-\de$ continuity is very weak in the sense that all standard functions are continuous on domains, where they are defined.
\medskip

For example, $f(x)=\dfrac{1}{x}$ is continuous for $x\neq 0$ because, for any $\ep>0$, there is $\de=\dfrac{|x|}{2}\min\{1,\ep|x|\}$ such that if $|x-y|<\de$, then $|f(x)-f(y)|=\dfrac{|x-y|}{|xy|}\leq\dfrac{\de}{|xy|}\leq\dfrac{2\de}{x^2}\leq\ep$.
However, for small $x>0$, the chosen delta $\de$ is much smaller than $\ep$, so $f(x)=\dfrac{1}{x}$  grows too fast close to $0$.
The Lipschitz continuity below is more practical by restricting the growth of a function via a constant and an amount of perturbation.
\medskip

\begin{dfn}[Lipschitz continuity]
\label{dfn:Lipschitz continuity}
A function $f:\R^n\to\R$ is \emph{Lipschitz continuous}
if there is a \emph{Lipschitz constant} $\la>0$ such that
$|f(x)-f(y)|\leq\la|x-y|$ for any $x,y\in\R^n$, where $|x-y|$ denotes the Euclidean distance.
\end{dfn}
\medskip

Then $f(x)=\dfrac{1}{x}$ is not Lipschitz continuous because for any $\la>0$, we can set $c=\max\{1,\la\}$, $x=\dfrac{1}{2c}$ and $y=\dfrac{1}{c}$ such that $|f(x)-f(y)|=c\geq 1>\dfrac{\la}{2c}=\la|x-y|$.
\medskip

Though the Lipschitz continuity makes sense for maps between arbitrary metric spaces, we consider only scalar functions on subsets of $\R^n$.
\medskip

A simplex can be defined as a finite set of elements whose every subset is also a simplex.
We consider only geometric realizations of a simplex, still called a simplex. 
\medskip

\begin{dfn}[a geometric \emph{simplex} $T$ on any $n+1$ points in $\R^n$]
\label{dfn:geometric_simplex}
\textbf{(a)}
The (geometric) \emph{simplex} $T$ on any $n+1$ ordered points $p_0,\dots,p_n\in\R^n$ is the subset
$A=\left\{ \sum\limits_{i=0}^n t_i p_i \; \vline \; t_i\in[0,1],\;  \sum\limits_{i=0}^n t_i=1\right\}\subset\R^n$ with ordered vertices $p_0,\dots,p_n$.
\medskip

\noindent
\textbf{(b)}
An \emph{orientation} of a simplex $T$ is the sign of the determinant of the $n\times n$ matrix with the columns $p_1-p_0,\dots,p_n-p_0$, and is denoted by $\sign(T)$. 
\end{dfn}
\medskip

For $n=1$, the simplex on any points $p_0,p_1\in\R$ is the line segment connecting $p_0$ and $p_1$.
If points $p_0,\dots,p_n\in\R^n$ are \emph{affinely independent}, i.e. there is no $(n-1)$-dimensional affine subspace of $\R^n$ containing all $p_0,\dots,p_n$ and hence the simplex $T$, then $T$ is $n$-dimensional.
However, Definition~\ref{dfn:geometric_simplex} makes sense for any points. 
\medskip

The \emph{volume} $\vol(T)$ of a simplex $T$ or an arbitrary polyhedron is often used as a shape descriptor and detects affine independence in the sense that $T$ is \emph{degenerate} if and only if $\vol(T)=0$.
However, the volume and all other distance-based descriptors do not distinguish mirror images, which have different signs of orientation.
\medskip

When $T$ goes through a degenerate configuration, an orientation of $T$ can discontinuously change the sign.
This discontinuity is an obstacle to recognizing simplices (or more general clouds) that are nearly mirror-symmetric.
\medskip

One attempt to resolve this discontinuity is to consider the signed volume $\sign(T)\vol(T)$, because $\vol(T)$ vanishes only on degenerate simplices.
Unfortunately, $\vol(T)$ is not Lipschitz continuous in any dimension $n\geq 2$, as illustrated below.
\medskip

\begin{exa}[the area of a triangle is not Lipschitz continuous]
\label{exa:area_not_Lipschitz}
For any large $l>0$ and real $\ep$ close to $0$, let $T(l,\ep)\subset\R^2$ be the 2D simplex (triangle) on the vertices $(0,\ep)$ and $(\pm l,0)$.
The signed area $l\ep$ of $T(l,\ep)$ distinguishes mirror images $T(l,\pm\ep)$ but is not Lipschitz continuous under perturbations.
Indeed, as $\ep\to 0$, the triangle $T(l,\ep)$ degenerates to a straight line, while the area drops to 0 too quickly so that $\dfrac{\vol(T(l,\ep)) - \vol(T(l,0))}{\ep-0}=\dfrac{l\ep}{\ep}=l$ is not bounded.
Hence, if given points are not restricted to a fixed bounded region, a small change in their positions may lead to a large change in the area of a triangle, and similarly for the volume in $\R^n$.
\end{exa}
\medskip

\begin{pro}[Lipschitz continuous detection of degenerate simplices]
\label{pro:strength_simplex}
Find a Lipschitz continuous real-valued function $f(T)$ for all simplices $T$ on $n+1$ points $p_0,\dots,p_n\in\R^n$ such that 
$f(T)=0$ if and only if $T$ is degenerate.
\end{pro}


\section{The strength is a Lipschitz continuous invariant}
\label{sec:strength_simplex}

We solve Problem~\ref{pro:strength_simplex} by introducing the strength function in Definition~\ref{dfn:strength_simplex} and proving its Lipschitz continuity in Theorem~\ref{thm:strength_properties}.
\medskip

\begin{dfn}[the \emph{strength} $\si(T)$ and \emph{signed strength} $s(T)$ of a simplex $T\subset\R^n$]
\label{dfn:strength_simplex}
Let $T\subset\R^n$ be the simplex on any $n+1$ points $p_0,p_1,\dots,p_n$ in $\R^n$.
The \emph{half-perimeter} $p(T)=\dfrac{1}{2}\sum\limits_{i\neq j}|p_i-p_j|$ is one half of the sum of all distances between the vertices of $T$. 
The \emph{strength} of $T$ is $\si(T)=\dfrac{\vol^2(T)}{p^{2n-1}(T)}$.
The \emph{signed strength} is $s(T)=\sign(T)\si(T)$. 
\end{dfn}
\medskip

\begin{exa}[the strength of a line segment]
\label{exa:strength_segment}
For $n=1$, the simplex $T$ on any points $p_0,p_1\in\R$ is the line segment  with $\vol(T)=|p_1-p_0|=2p(T)$, the strength $\si(T)=\dfrac{\vol^2(T)}{p(T)}=2|p_1-p_0|$, and the signed strength $s(T)=2(p_1-p_0)$.
\end{exa}
\medskip

\begin{exa}[the strength of a triangle]
\label{exa:strength_triangle}
\textbf{(a)} 
Let $T\subset\R^2$ be a triangle with sides $a,b,c$.
Heron's formula for the area $\vol(T)=\sqrt{p(p-a)(p-b)(p-c)}$, where $p=\dfrac{a+b+c}{2}=p(T)$ is the half-perimeter, gives the strength $\si(T)=\dfrac{\vol^2(T)}{p^3(T)}=\dfrac{(p-a)(p-b)(p-c)}{p^2}$. 
The triangle $T(l,\ep)$ in Example~\ref{exa:area_not_Lipschitz} has $p=\dfrac{l+\ep+\sqrt{l^2+\ep^2}}{2}$.
Using $l\leq\sqrt{l^2+\ep^2}\leq l+\ep$, we can estimate the strength of $T(l,\ep)$ as follows: $$\si=\dfrac{(l-\ep+\sqrt{l^2+\ep^2})(\ep+\sqrt{l^2+\ep^2}-l)(l+\ep-\sqrt{l^2+\ep^2})}{2(l+\ep+\sqrt{l^2+\ep^2})^2}\leq  
\dfrac{2l\cdot 2\ep\cdot \ep}{2(2l)^2}\leq \dfrac{\ep^2}{2l}\leq\dfrac{\ep}{2}$$ for any $0\leq\ep\leq l$.
Hence, the strength of $T(l,\ep)$ is Lipschitz continuous. 
\medskip
 
\noindent
\textbf{(b)}
To visualize the strength $\si(T)$ of a triangle $T$, we assume that $0<a\leq b\leq c$ and normalize $T$ by the larger side $c$ to get the sides $\ti a=\dfrac{a}{c}\leq \ti b=\dfrac{b}{c}\leq \ti c=1$.
The resulting space of normalized triangles is parameterized by the coordinates $x=\dfrac{a}{c}$ and $y=1-\dfrac{b}{c}$ in the triangular region $\De=\{(x,y)\in\R^2 \mid x\in[0,1],\, x\geq y,\, x+y\leq 1\}$, where $x+y\leq 1$ means that $a\leq b$, while $x\geq y$ is equivalent to the triangle inequality $a+b\geq c$, see Fig.~\ref{fig:strength_plots}~(left).
The half-perimeter is $\ti p=\dfrac{1}{2}(\ti a+\ti b+\ti c)=\dfrac{1}{2}(x+(1-y)+1)=\dfrac{1}{2}(2+x-y)$.
Then $\ti p-\ti a=\dfrac{1}{2}(2-x-y)$,   
$\ti p-\ti b=\dfrac{1}{2}(x+y)$,
$\ti p-\ti c=\dfrac{1}{2}(x-y)$, and the strength is
$\si(x,y)=\dfrac{(2-x-y)(x^2-y^2)}{2(2+x-y)^2}$, see Fig.~\ref{fig:strength_plots}~(right).
\medskip

In the triangular region $\De$ in  Fig.~\ref{fig:strength_plots}~(left), the horizontal side $\{x\in(0,1],\, y=0\}$ represents all (normalized) isosceles triangles with $\ti a\leq \ti b=\ti c=1$ and strength $\si=\dfrac{(2-x)x^2}{2(2+x)^2}$ for $x\in[0,1]$.
The right hand side $\{x\in(0.5,1],\, x+y=1\}$ of the region $\De$ represents all (normalized) isosceles triangles with $\ti a=\ti b\leq \ti c=1$ and strength $\si=\dfrac{2x-1}{2(2x+1)^2}$ for $x\in[\frac{1}{2},1]$.
The vertex $(x,y)=(1,0)$ represents all equilateral triangles with side 1 of strength $\si=\dfrac{1}{18}$.
Any equilateral triangle $T$ with sides $a=b=c$ has the strength $\si(T)=\dfrac{(a^2\sqrt{3}/4)^2}{(3a/2)^3}=\dfrac{a}{18}$.
\end{exa}
\medskip

\begin{figure}[h!]
\includegraphics[height=34mm]{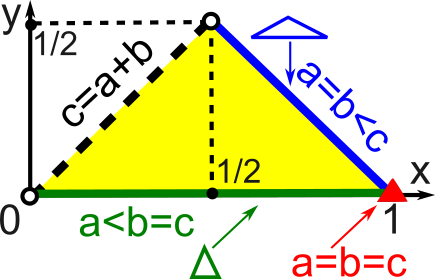}
\includegraphics[height=34mm]{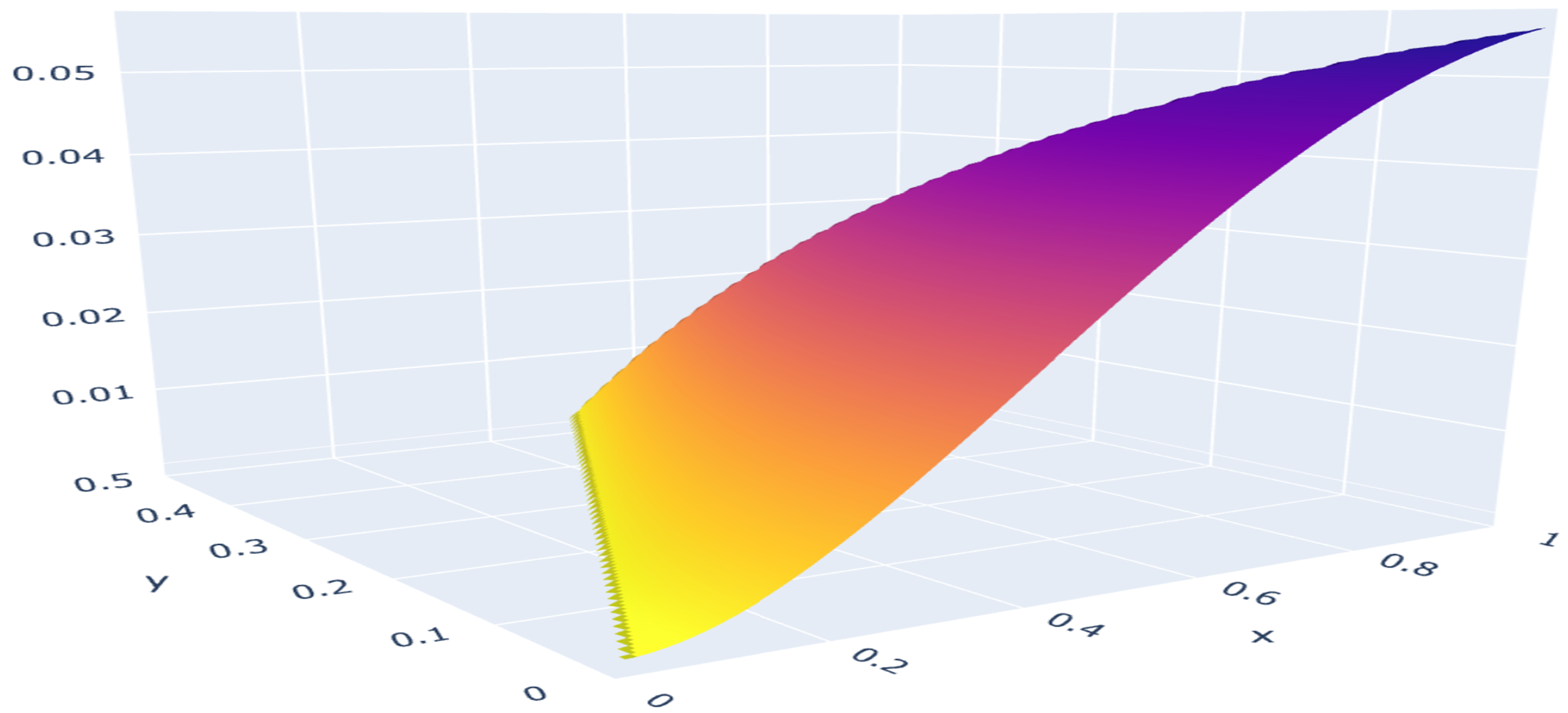}
\label{fig:strength_plots}
\caption{
\textbf{Left}: 
Example~\ref{exa:strength_triangle}(b) parameterizes the space of (normalized) triangles with sides $0<a\leq b\leq c$ by $x=\dfrac{a}{c}$ and $y=1-\dfrac{b}{c}$ over the region $\De=\{(x,y)\in\R^2 \mid x\in[0,1],\, x\geq y,\, x+y\leq 1\}$.
\textbf{Right}: 
the strength
$\si(x,y)=\dfrac{(2-x-y)(x^2-y^2)}{2(2+x-y)^2}$ of a normalized triangle over the region $\De$.} 
\end{figure}

\begin{figure}[h!]
\includegraphics[height=36mm]{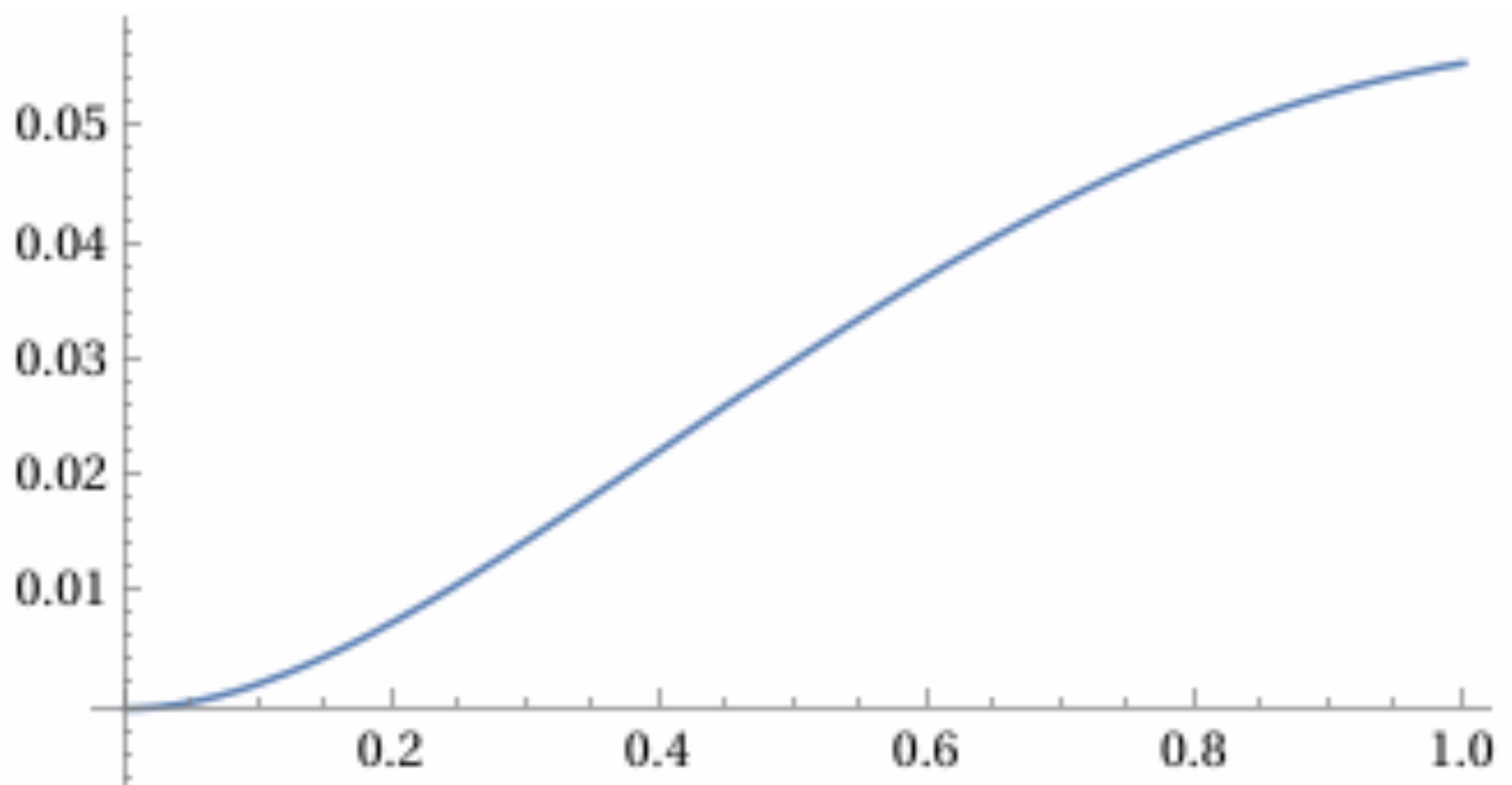}
\includegraphics[height=36mm]{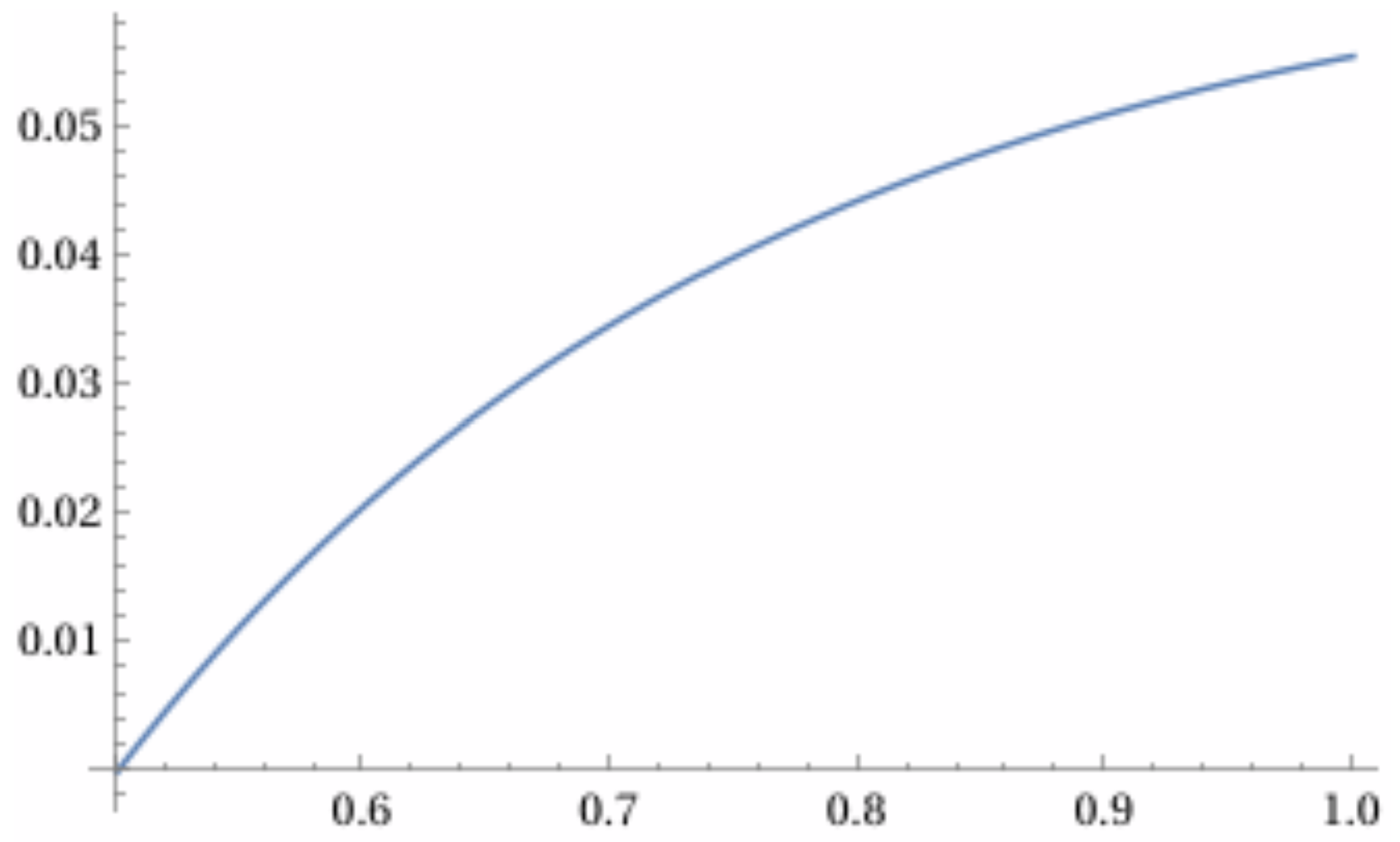}
\label{fig:strength_curves}
\caption{
Vertical sections of the strength $\si$ of normalized triangles parameterized by $x,y$ in Fig.~\ref{fig:strength_plots}.
\textbf{Left}: isosceles triangles with $\ti a=\ti b\leq \ti c=1$ have $y=0$ and $\si=\dfrac{(2-x)x^2}{2(2+x)^2}$ for $x\in[0,1]$. 
\textbf{Right}: isosceles triangles with sides $\ti a\leq \ti b=\ti c$ have $y=1-x$ and $\si=\dfrac{2x-1}{2(2x+1)^2}$ for $x\in[\frac{1}{2},1]$. }
\end{figure}

Recall that an \emph{isometry} is any distance-preserving transformation of $\R^n$, which decomposes into translations and orthogonal maps from the orthogonal group $\Or(\R^n)$.
\medskip

A \emph{rigid motion} is any composition of translations and rotations from the special orthogonal group $\SO(\R^n)$.
The strength of a simplex was essentially used to define a Lipschitz continuous metric on invariants of $n$-dimensional clouds of $m$ unordered points, which are complete under rigid motion in $\R^n$ and can be computed in a polynomial time of $m$, for a fixed dimension $n$ \cite[Theorem~4.7]{widdowson2023recognizing}, see details in \cite{kurlin2023strength}.
\medskip

All time complexities are considered in the RAM model of computation so that any value in computer memory can be accessed in a constant time.
\medskip

\begin{thm}[strength properties: invariance, complexity, and Lipschitz continuity]
\label{thm:strength_properties}
\textbf{(a)}
The strength $\si(T)$ and signed strength $s(T)$ of any simplex $T\subset\R^n$ are invariant under isometry and rigid motion in $\R^n$, respectively, and can be computed in time $O(n^3)$.
The uniform scaling of $\R^n$ by any factor $c>0$ multiplies $\si(T)$ and $s(T)$ by $c$.
\medskip

\noindent
\textbf{(b)}
Fix any dimension $n\geq 1$.
Then there is a constant $\la_n>0$ such that, for any $\ep>0$, if a simplex $Q$ is obtained from another simplex $T\subset\R^n$ by perturbing every vertex of $T$ within its $\ep$-neighborhood, then $|\si(T)-\si(Q)|\leq 2\la_n\ep$ and $|s(T)-s(Q)|\leq 2\la_n\ep$, where $\la_1=2$, $\la_2=\sqrt{3}$, and
$\la_n\leq \dfrac{2^{n+0.5}}{(n!) n^{2n-4}}$ for any $n\geq 3$, e.g. $\la_3<0.37$.
\end{thm}

\section{Proofs of strength properties in all dimensions}
\label{sec:proofs}

This section proves Theorem~\ref{thm:strength_properties} by using the Cayley-Menger determinant below.

\begin{dfn}[Cayley-Menger determinant, {\cite[Chapter II, p.98]{blumenthal1970theory}}]
\label{dfn:Cayley-Menger determinant}
Let the simplex $T$ on any points $p_0,\dots,p_{n}\in\R^n$ have the $(n+1)\times(n+1)$ matrix $D_{ij}$ of squared Euclidean distances $d_{ij}^2=|p_i-p_j|^2$ for $i,j=0,\dots,n$.
The $(n+2)\times(n+2)$ matrix $\hat D$ is obtained from $D$ by adding the top row $(0,1,\dots,1)$ and the left column $(0,1,\dots,1)^T$.
The \emph{Cayley-Menger determinant}
 \cite{cayley1841theorem,menger1931new} expresses the squared volume of the simplex $T$ as $\vol^2(T)=\dfrac{(-1)^{n-1}}{2^n(n!)^2}\det\hat D$.
\end{dfn}

\begin{proof}[\textbf{Proof of Theorem~\ref{thm:strength_properties}(a)}]
Any isometry preserves all distances and hence the strength $\si(T)$ expressed via distances in Definition~\ref{dfn:strength_simplex}.
Any special orthogonal matrix $M\in\SO(\R^n)$  keeps the sign of a simplex $T$, i.e. multiplies $\sign(T)$ by $\det M=1$.
Then any rigid motion preserves the signed strength $s(T)=\sign(T)\si(T)$.
To compute $\si(T)$ and $s(T)$, we need the half-perimeter $p(T)$, which requires a quadratic time $O(n^2)$, and then $\sign(T)$ and $\vol^2(T)$ by using determinants of sizes up to $n+2$, which can be calculated in time $O(n^3)$ by matrix diagonalization \cite[section~11.5]{press2007numerical}.
\end{proof}

We prove Theorem~\ref{thm:strength_properties}(b) first for the dimensions $n=1,2$  and then for $n\geq 3$.

\begin{proof}[\textbf{Proof of Theorem~\ref{thm:strength_properties}(b) for $n=1$ and $\la_1=2$}]
For $n=1$, a simplex $T\subset\R$ has two vertices $p_0,p_1$ at a distance $d=|p_0-p_1|$.
Then the strength $\si(T)=2d$ and $s(T)=2(p_1-p_0)$ have the Lipschitz constant $2\la_1=4$.
Indeed, perturbing each of $p_0,p_1$ up to $\ep$ changes the difference $p_1-p_0$ and the distance $d$ up to $2\ep$. 
\end{proof}

Vertical lines are used to denote the determinant $|M|$ of a matrix $M$, the Euclidean length $|v|$ of a vector $v\in\R^n$, and the absolute value $|r|$ of a real number $r$.

\begin{proof}[\textbf{Proof of Lipschitz continuity for $\si$ in Theorem~\ref{thm:strength_properties}(b) for $n=2$, $\la_2=\sqrt{3}$}]
Let a triangle $T\subset\R^2$ have pairwise distances $a,b,c$.
Using the half-perimeter $p=\dfrac{a+b+c}{2}$, the variables $\ti a=p-a$, $\ti b=p-b$, $\ti c=p-c$ are 
expressed via $a,b,c$, so $a=\ti b+\ti c$, $b=\ti a+\ti c$, $c=\ti a+\ti b$, $p=\ti a+\ti b+\ti c$.
The Jacobian of this change of variables is $\left|\bpd{(a,b,c)}{(\ti a,\ti b,\ti c)}\right|=\left| \begin{array}{ccc} 0 & 1 & 1 \\ 1 & 0 & 1 \\ 1 & 1 & 0 \end{array} \right|=2$. 
If every point of $T$ is perturbed up to $\ep$, then any pairwise distance between the vertices of $T$ changes by at most $2\ep$. 
\smallskip

By the mean value theorem \cite{elliott2012probabilistic}, this bound $2\ep$ gives $|\si(T)-\si(Q)|\leq 2\ep\sup|\nabla\si|$, where $\nabla\si=\left(\bpd{\si}{a},\bpd{\si}{b},\bpd{\si}{c}\right)$ is the gradient of the first order partial derivatives of $\si(T)$ with respect to the three distances between points of $T$.
\smallskip

Since $|\nabla\si|=\left|\bpd{(a,b,c)}{(\ti a,\ti b,\ti c)}\cdot
\left(\bpd{\si}{\ti a},\bpd{\si}{\ti b},\bpd{\si}{\ti c}\right)\right|\leq 2\left|\left(\bpd{\si}{\ti a},\bpd{\si}{\ti b},\bpd{\si}{\ti c}\right)\right|$, it remains to estimate the first order partial derivatives of 
$\si=\dfrac{\ti a\ti b\ti c}{(\ti a+\ti b+\ti c)^2}$
 with respect to the variables $\ti a,\ti b,\ti c$.   
Since $\si$ is symmetric in $\ti a,\ti b,\ti c$, it suffices to consider
$$\begin{array}{l}
\bpd{\si}{\ti a}=\dfrac{\ti b\ti c}{(\ti a+\ti b+\ti c)^2}-\dfrac{2\ti a\ti b\ti c}{(\ti a+\ti b+\ti c)^3}=\dfrac{\ti b\ti c(\ti b+\ti c-\ti a)}{(\ti a+\ti b+\ti c)^3}=\\ \\
\dfrac{(p-b)(p-c)(2a-p)}{p^3}=\left(1-\dfrac{b}{p}\right)\left(1-\dfrac{c}{p}\right)\left(2\dfrac{a}{p}-1\right).
\end{array}$$

By the triangle inequalities for the sides $a,b,c$, we have $\max\{a,b,c\}\leq p=\dfrac{a+b+c}{2}$.
Then $\dfrac{a}{p},\dfrac{b}{p},\dfrac{c}{p}\in(0,1]$, and
$1-\dfrac{b}{p},1-\dfrac{c}{p}\in[0,1)$, and 
$2\dfrac{a}{p}-1\in (-1,1]$, so 
$\left|\bpd{\si}{\ti a}\right|\leq 1$.
The similar bounds $\left|\bpd{\si}{\ti b}\right|,\left|\bpd{\si}{\ti c}\right|\leq 1$ imply that $|\nabla\si|\leq 2\sqrt{1^2+1^2+1^2}= 2\sqrt{3}$, so $|\si(T)-\si(Q)|\leq 2\sup|\nabla\si|\ep \leq 2\la_2\ep$ for $\la_2=\sqrt{3}$.
\end{proof}
\medskip

Theorem~\ref{thm:strength_properties}(b) for any $n\geq 3$ will be proved by Lemmas~\ref{lem:edge_ratios}, \ref{lem:determinant_ratio}, and~\ref{lem:determinant_derivative}.
\medskip

\begin{lem}[edge ratios]
\label{lem:edge_ratios}
For any simplex $T$ on points $p_0,\dots,p_n\in\R^n$, we have $\dfrac{|p_i-p_j|}{p(T)}\leq\dfrac{2}{n}$ for all $i,j=0,\dots,n$. 
\end{lem}  
\begin{proof}
Set $d_{ij}=|p_i-p_j|$ and use triangle inequalities as follows: 
$2p(T)=\sum\limits_{k,l=0}^{n} d_{kl}\geq d_{ij}+\sum\limits_{k\neq i,j} (d_{ki}+d_{kj})\geq d_{ij}+(n-1)d_{ij}=nd_{ij}$.
\end{proof}
\medskip

Recall that the \emph{rencontre} number $r_n=n!\sum\limits_{k=0}^n\dfrac{(-1)^k}{k!}$ counts all permutations of $1,\dots,n$ without a fixed point \cite{charalambides2018enumerative} and equals the integer nearest to $\dfrac{n!}{e}$, e.g. $r_2=1$, $r_3=2$, $r_4=9$, $r_5=44$.

\begin{lem}
\label{lem:determinant_ratio}
In the notations of Definition~\ref{dfn:Cayley-Menger determinant}, we have
$\dfrac{\det\hat D}{p^{2n}(T)}\leq r_{n+2}\left(\dfrac{2}{n}\right)^{2n}$ for any $n\geq 1$.
\end{lem}
\begin{proof}
The determinant formula $\det\hat D=\sum\limits_{\xi\in S_{n+2}} (-1)^{\sign(\xi)}\hat D_{1,\xi(1)}\dots\hat D_{n,\xi(n)}$ over all permutations $\xi\in S_{n+2}$ excludes all zeros on the diagonal.
Then $k\neq\xi(k)$ for $k=0,\dots,n+1$, i.e. we can consider only permutations $\xi$ of $0,\dots,n+1$ that have no fixed elements.
The number of such permutations is the rencontre number $r_{n+2}$.
Then $\det\hat D$ is a sum of $r_{n+2}$ non-zero terms, each being a product of $n$ squared distances $d_{ij}$.
The upper bound $d_{ij}\leq\dfrac{2}{n}p(T)$ in Lemma~\ref{lem:edge_ratios} implies that each term in the sum of $\det\hat D$ is at most $\left(\dfrac{2}{n}p(T)\right)^{2n}$.
Then $\dfrac{\det\hat D}{p^{2n}(T)}\leq r_{n+2}\left(\dfrac{2}{n}\right)^{2n}$ as required.
\end{proof}

For $n=1$ in Lemma~\ref{lem:determinant_ratio}, we have $\det\hat D=2d_{01}^2$ and $p(T)=\dfrac{d_{01}}{2}$, so $\dfrac{\det\hat D}{p^{2}(T)}=8\leq r_3(2^2)=8$ as expected.
\medskip

\begin{lem}
\label{lem:determinant_derivative}
In the notations of Definition~\ref{dfn:Cayley-Menger determinant}  for any distinct indices $i,j\in\{1,\dots,n\}$, we have that 
$\left|\bpd{\det\hat D}{d_{ij}}\right|\dfrac{1}{p^{2n-1}(T)}
\leq 4(r_n+r_{n+1}) \left(\dfrac{2}{n}\right)^{2n-1}$.
\end{lem}
\begin{proof}
Since $\det\hat D$ has $d_{ij}^2$ in exactly two cells in different rows and columns, $\det\hat D$ is a quadratic polynomial $\al d_{ij}^4+\be d_{ij}^2+\ga$ for some $\al,\be,\ga$ that depend on other fixed distances $d_{kl}\neq d_{ij}$.
Then $\bpd{\det\hat D}{d_{ij}}=4\al d_{ij}^3+2\be d_{ij}$.
The coefficient $\al$ is the determinant of the $n\times n$ submatrix $(b_{ij})$ obtained from $\hat D$ by removing two rows and columns indexed by $i+2,j+2$.
For example, fix $i=0$ and $j=1$.
If $n=2$, then $\al=\left| \begin{array}{cc} 
0 & 1 \\
1 & 0 
\end{array}\right|=-1$.
If $n=3$,  then $\al=\left| \begin{array}{ccc} 
0 & 1 & 1 \\
1 & 0 & d_{23}^2 \\
1 & d_{32}^2 & 0  
\end{array}\right|=2d_{23}^2$.
Since the matrix $(b_{ij})$ has zeros on the main diagonal, its determinant $\al=\sum\limits_{\xi\in S_n} (-1)^{\sign(\xi)}b_{1,\xi(1)}\dots b_{n,\xi(n)}$ is a sum over all permutations $\xi\in S_n$ 
with no fixed points.
Then the sum $\al$ has $r_n$ non-zero products $b_{1,\xi(1)}\dots b_{n,\xi(n)}$ and the total degree $2(n-2)$ in all distances $d_{kl}\neq d_{ij}$.
After dividing the polynomial $\al d_{ij}^3$ of the degree $2n-1$ by $p^{2n-1}(T)$,
we use the upper bound $\dfrac{d_{kl}}{p(T)}\leq\dfrac{2}{n}$ in Lemma~\ref{lem:edge_ratios} to get $\dfrac{|\al|d_{ij}^3}{p^{2n-1}(T)}\leq r_n\left(\dfrac{2}{n}\right)^{2n-1}$.
\smallskip

The coefficient $\be$ in $\det\hat D=\al d_{ij}^4+\be d_{ij}^2+\ga$ is the sum of products from the determinants of the two 
submatrices obtained from $\hat D$ by removing row $i+2$ and column $j+2$ (for one submatrix), then row $j+2$ and column $i+2$ (for another submatrix).
If $n=3$, $i=0$, $j=2$, the determinants are
$\left| \begin{array}{cccc} 
0 & 1 & 1 & 1  \\
1 & d_{10}^2 & 0 & d_{13}^2 \\
1 & d_{20}^2 & d_{21}^2 & d_{23}^2 \\
1 & d_{30}^2 & d_{32}^2 & 0 
\end{array} \right|+
\left| \begin{array}{cccc} 
0 & 1 & 1 & 1 \\
1 & d_{01}^2 & d_{02}^2 & d_{03}^2 \\
1 & 0 & d_{12}^2 & d_{13}^2 \\
1 & d_{31}^2 & d_{32}^2 & 0 
\end{array} \right|$.
\smallskip

Since each $(n+1)\times(n+1)$ submatrix includes one entry $d_{ij}^2$, we exclude all products with this entry, which are multiplied by the removed $d_{ij}^2$ from $\hat D$ and hence were counted in $\al d_{ij}^4$.
Hence we can replace $d_{02}=d_{02}$ with 0 and get
$\be=\left| \begin{array}{cccc} 
0 & 1 & 1 & 1  \\
1 & d_{10}^2 & 0 & d_{13}^2 \\
1 & 0 & d_{21}^2 & d_{23}^2 \\
1 & d_{30}^2 & d_{32}^2 & 0 
\end{array} \right|+
\left| \begin{array}{cccc} 
0 & 1 & 1 & 1 \\
1 & d_{01}^2 & 0 & d_{03}^2 \\
1 & 0 & d_{12}^2 & d_{13}^2 \\
1 & d_{31}^2 & d_{32}^2 & 0 
\end{array} \right|$.
Up to a permutation of indices, each submatrix can be rewritten with the diagonal that has one original $d_{ij}^2$ (now replaced with 0), while all other diagonal elements are initial zeros. 
The example above gives
$\be=-\left| \begin{array}{cccc} 
0 & 1 & 1 & 1  \\
1 & 0 & d_{21}^2 & d_{23}^2 \\
1 & d_{10}^2 & 0 & d_{13}^2 \\
1 & d_{30}^2 & d_{32}^2 & 0 
\end{array} \right|-
\left| \begin{array}{cccc} 
0 & 1 & 1 & 1 \\
1 & 0 & d_{12}^2 & d_{13}^2 \\
1 & d_{01}^2 & 0 & d_{03}^2 \\
1 & d_{31}^2 & d_{32}^2 & 0 
\end{array} \right|$.
Similarly to the argument for the determinant $\al$,
the sum $\be d_{ij}$ contains $2r_{n+1}$ products of the total degree $2n-1$, so $\dfrac{|\be| d_{ij}}{p^{2n-1}(T)}\leq 2r_{n+1}\left(\dfrac{2}{n}\right)^{2n-1}$ by Lemma~\ref{lem:edge_ratios}.
Then the required inequality follows:
$\left|\bpd{\det\hat D}{d_{ij}}\right|\dfrac{1}{p^{2n-1}(T)}=\dfrac{|4\al d_{ij}^3+2\be d_{ij}|}{p^{2n-1}(T)}\leq (4r_n+4r_{n+1}) \left(\dfrac{2}{n}\right)^{2n-1}$.
\end{proof}
\medskip

\begin{proof}[\textbf{Proof of Theorem~\ref{thm:strength_properties}(b) for $n\geq 3$}]
For $n=3$, the squared volume is 
$\vol^2(T)=\dfrac{1}{288}
\left| \begin{array}{ccccc} 
0 & 1 & 1 & 1 & 1 \\
1 & 0 & d_{01}^2 & d_{02}^2 & d_{03}^2 \\
1 & d_{10}^2 & 0 & d_{12}^2 & d_{13}^2 \\
1 & d_{20}^2 & d_{21}^2 & 0 & d_{23}^2 \\
1 & d_{30}^2 & d_{31}^2 & d_{32}^2 & 0 
\end{array} \right|$.
Similarly to the case $n=2$, the mean value theorem
\cite{elliott2012probabilistic} for the strength $\si(T)=\dfrac{\vol^2(T)}{p^{2n-1}(T)}$ 
implies that
$|\si(T)-\si(Q)|\leq 2\ep\sup|\nabla\si|\leq 
2\ep\sqrt{ \sum\limits_{i\neq j}\sup\left|\bpd{\si}{d_{ij}}\right|^2}\leq 
2\ep\sqrt{\dfrac{n(n+1)}{2}}\max\limits_{i\neq j}\sup\left|\bpd{\si}{d_{ij}}\right|.$
To find an upper bound of $\left|\bpd{\si}{d_{ij}}\right|$, 
we initially ignore the numerical factor in the square volume $\vol^2(T)=\dfrac{(-1)^{n-1}}{2^n(n!)^2}\det\hat D$ and differentiate $\det\hat D\cdot\dfrac{1}{p^{2n-1}(T)}$ by the product rule:
$$\bpd{}{d_{ij}}\left(\dfrac{\det\hat D}{p^{2n-1}(T)}\right)=
\bpd{\det\hat D}{d_{ij}}\cdot\dfrac{1}{p^{2n-1}(T)}-\dfrac{\det\hat D}{p^{2n}(T)}\cdot\dfrac{2n-1}{2}.$$
Lemmas~\ref{lem:determinant_ratio} and \ref{lem:determinant_derivative} imply the upper bound 
$$\left|\bpd{}{d_{ij}}\left(\dfrac{\det\hat D}{p^{2n-1}(T)}\right)\right|
\leq(4r_n+4r_{n+1}) \left(\dfrac{2}{n}\right)^{2n-1}
+\left(n-\dfrac{1}{2}\right)r_{n+2}\left(\dfrac{2}{n}\right)^{2n}<$$
$<\left(4r_n+4r_{n+1}+2r_{n+2}\right)\left(\dfrac{2}{n}\right)^{2n-1}$.
Taking into account the factors $\dfrac{(-1)^{n-1}}{2^n(n!)^2}$ in $\vol^2(T)$ and $\sqrt{\dfrac{n(n+1)}{2}}$ for estimating the length of the gradient $\nabla\si$ of $\dfrac{n(n+1)}{2}$ first order partial derivatives $\bpd{\si}{d_{ij}}$, 
we get $$|\si(T)-\si(Q)|\leq 2\ep\sqrt{\dfrac{n(n+1)}{2}}\max\limits_{i\neq j}\sup\left|\bpd{\si}{d_{ij}}\right|\leq 2\ep c_n$$ for
the upper bound 
$$\begin{array}{l}
c_n=
\dfrac{\sqrt{n(n+1)/2}}{2^n(n!)^2}\left(4r_n+4r_{n+1}+2r_{n+2}\right)\left(\dfrac{2}{n}\right)^{2n-1}=\\
=(2r_n+2r_{n+1}+r_{n+2})\dfrac{2^{n-0.5}\sqrt{n+1}}{(n!)^2n^{2n-1.5}}.
\end{array}$$
The estimate $r_n\leq\dfrac{n!}{2}$ gives 
$$\begin{array}{l}
c_n<\dfrac{n!}{2}\Big(2+2(n+1)+(n+1)(n+2)\Big)\dfrac{2^{n-0.5}\sqrt{n+1}}{(n!)^2n^{2n-1.5}}= \\
=(n^2+5n+6)\sqrt{n+1}\dfrac{2^{n-1.5}}{(n!) n^{2n-1.5}}
=\left(1+\dfrac{5}{n}+\dfrac{6}{n^2}\right)\sqrt{1+\dfrac{1}{n}}\dfrac{2^{n-1.5}}{(n!) n^{2n-4}}.
\end{array}$$
For $n\geq 3$, we get $1+\dfrac{5}{n}+\dfrac{6}{n^2}
\leq\dfrac{10}{3}$ and $\sqrt{1+\dfrac{1}{n}}\leq\dfrac{2}{\sqrt{3}}$, so \\
$c_n<\dfrac{20}{3\sqrt{3}}\cdot \dfrac{2^{n-1.5}}{(n!) n^{2n-4}}= \dfrac{5\sqrt{2}}{3\sqrt{3}}\dfrac{2^{n}}{(n!) n^{2n-4}}
<\dfrac{2^{n+0.5}}{(n!) n^{2n-4}}=b_n.$ 
\end{proof}
\medskip

\begin{proof}[Proof of Lipschitz constants for the signed strength in Theorem~\ref{thm:strength_properties}(b)]
If simplices $T,Q$ have $\sign(T)=\sign(Q)$, then $|s(T)-s(Q)|=|\si(T)-\si(Q)|\leq 2\la_n\ep$.
\medskip

If simplices $T,Q$ have opposite signs, we will prove that $|s(T)-s(Q)|\leq 4\la_n\ep$.
Any simplex $T$ can be connected to its $\ep$-perturbation $Q$ by a straight-line deformation through simplices $Q_t\subset\R^n$, where $t\in[0,1]$, $Q_0=T$, and $Q_1=Q$.
\medskip

There is an intermediate value $t\in (0,1)$ such that $\si(Q_t)=0=\sign(Q_t)$.
Then both $T,Q$ are deformations of the intermediate simplex $Q_t$ such that all corresponding vertices are perturbed up to Euclidean distances $t\ep$ and $(1-t)\ep$, respectively.
The inequalities $|s(T)|=|\si(T)-\si(Q_t)|\leq 2\la_n t\ep$ and
$|s(Q)|=|\si(Q_t)-\si(Q)|\leq 2\la_n(1-t)\ep$ imply that $|s(T)-s(Q)|\leq |\si(T)-\si(Q_t)|+|\si(Q_t)-\si(Q)|\leq 2\la_n\ep$.
\end{proof}

In conclusion, the strength of a simplex provides a Lipschitz continuous analog of a volume, which is measured in the units of original point coordinates.
We thank Janos Pach for helpful discussions of this paper leading to forthcoming joint work. 
\medskip

Data availability statement.
All data for this research is included in the paper.

\bibliography{DCG2026strength}

\end{document}